\documentclass{amsart}

\usepackage{graphicx, amsmath, amsfonts, amssymb, hyperref, amsthm, comment, float, enumerate}
\usepackage[ruled,vlined]{algorithm2e}
\usepackage{tikz}

\newtheorem{theorem}{Theorem}[section]

\theoremstyle{definition}
\newtheorem{definition}[theorem]{Definition}
\newtheorem{example}[theorem]{Example}

\theoremstyle{remark}
\newtheorem{remark}[theorem]{Remark}

\numberwithin{equation}{section}

\begin{document}

\title[Thuston's Teapots and Graph Directed Systems]{Thurston's Teapots and Graph Directed Systems}


\author[Chenxi Wu]{Chenxi Wu}
\address{Department of Mathematics, UW Madison, Van Vleck Hall, Madison WI 53703}
\email{cwu367@wisc.edu}



\date{\today}

\begin{abstract}
Thurston's Master Teapot is a geometric object that encodes the entropies of critically periodic unimodal maps. We establish the connection between this object and the ``Mandelbrot set'' of graph directed iterated function systems previously studied by Solomyak.
\end{abstract}

\maketitle

\section{Introduction}

In one of his last papers \cite{thurston2014entropy}, William Thurston defined the concept of the ``master teapot'':

\begin{multline*}
   T=\overline{\{(z, \lambda): \lambda=e^{h_f}, f\text{ unimodal map with periodic critical orbit },}\\
\overline{h_f \text{ topological entropy of }f, z\text{ Galois conjugate of }\lambda\}}
\end{multline*}

\noindent Here a unimodal map is an interval map $f$ with a single critical point $c$, separating two subintervals, on one of which the map is strictly increasing while on the other the map is strictly decreasing. ``Periodic critical orbit'' means that there is some positive integer $n$ such that $f^{\circ n}(c)=c$. Because the points $\{c, f(c), \dots, f^{\circ(n-1)}(c)\}$ split the interval into finitely subintervals, and these splitting constitutes a finite Markov decomposition, the number $\lambda$ is the spectral radius of an integer matrix hence an algebraic integer, and its Galois conjugates are all well defined.

In \cite{bray2021shape, lindsey2023characterization}, by making use of the Milnor-Thurston kneading theory \cite{milnor1988iterated}, my collaborators and I developed an algorithm which can be used to certify points which are not on the master teapots. The purpose of this notes is to establish, via this algorithm, the connection between horizontal slices of the teapot and the ``Mandelbrot set'' associated to a parameterized family of graph directed iterated function systems and some constant $C$. Here we follow the convention of Solomyak, by ``Mandelbrot set'' we mean the set of parameters where the limit set of this graph directed iterated function system contains $C$, see Definition \ref{dfnman} part 3. Combining with the work of Solomyak \cite{solomyak2005mandelbrot}, this fact can be used to show self similarities around certain points of these slices (see Figure \ref{fig:slice}). Similar arguments also applies to the cases in \cite{lindsey2025master, wu2023concatenations}.

\begin{figure}[H]
    \centering
    \includegraphics[width=0.7\linewidth]{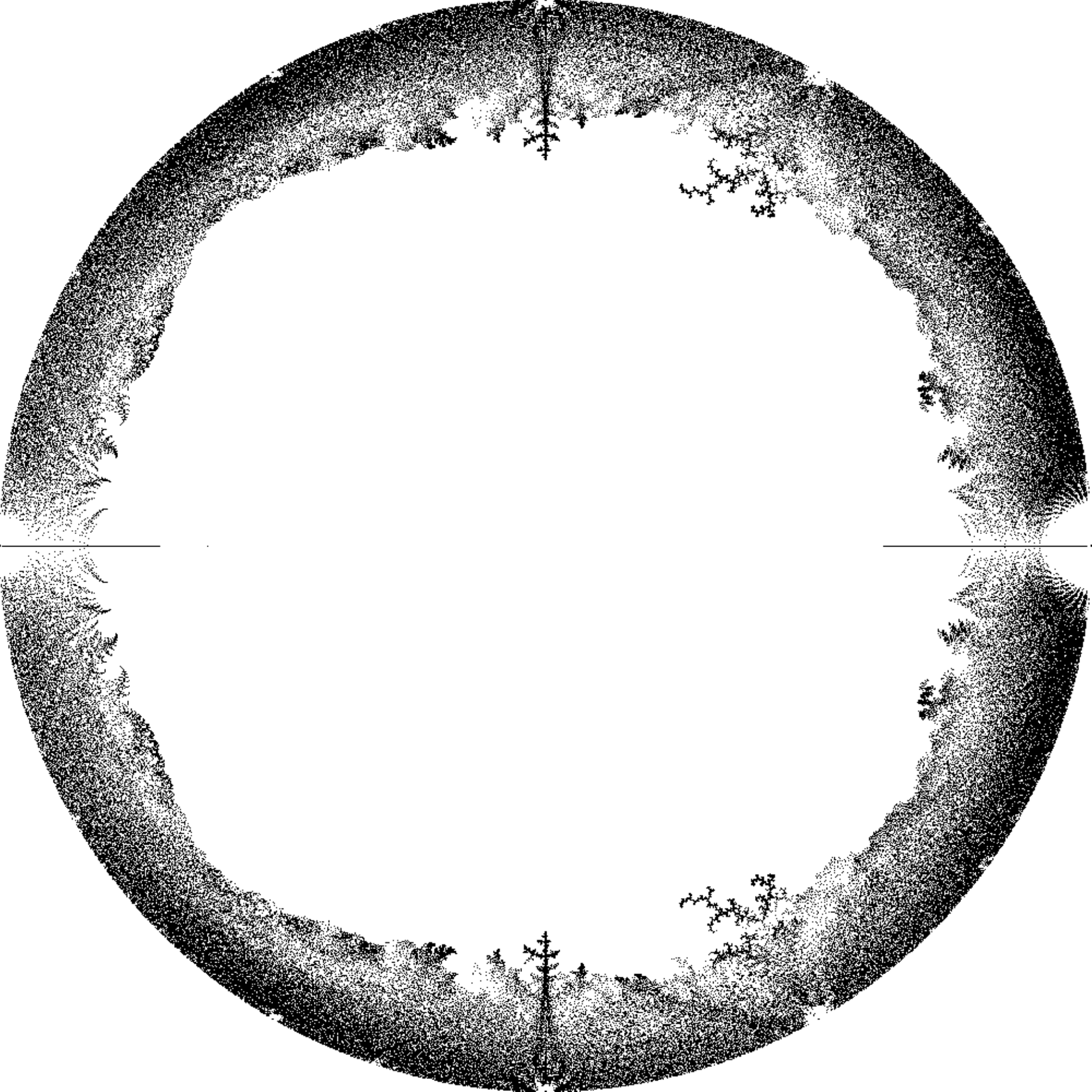}
    \caption{The horizontal slice at $\lambda=(\sqrt{5}+1)/2$ intersecting with the unit disc.}
    \label{fig:slice}
\end{figure}

\section{Graph Directed Iterated Function Systems}

Graph directed iterated function systems (GIFS) are generalizations of the concept of iterated function systems (IFS):

\begin{definition}\label{dfnman}\hspace{0pt}
\begin{enumerate}
    \item By a {\em family of graph directed iterated function systems}, we mean a finite directed graph $G$, where every edge $e$ is associated with a holomorphic function $f_e(z, x)$ from $\mathbb{C}^2$ to $\mathbb{C}$, and when $|z|<1$ it is contracting with respect to $x$.
    \item The {\em limit set} $\Lambda_z$ of a graph directed iterated function system with parameter $z$, is the closure of the points of the form 
    \[\lim_{n\to\infty}f_{e_0}(z, f_{e_1}(z, f_{e_2}(z,\dots f_{e_n}(z, 0)\cdot)))\]
    Where $e_0, e_1, e_2, \dots$ is any infinite sequence of consecutive edges in the graph starting at a base vertex $v_0$.
    \item The {\em Mandelbrot set} of this parametric family of graph directed iterated function systems for some constant $C$ is the set of all $z$ such that $C\in \Lambda_z$.
\end{enumerate}
\end{definition}

\begin{definition}\label{tent}
    Let $1<\lambda<2$. By a {\em tent map with slope $\lambda$} we mean the interval map from $[0, 1]$ to itself defined as 
    \[f_\lambda(x)=\begin{cases}\lambda x+2-\lambda & x\leq 1-1/\lambda\\ \lambda-\lambda x & x>1-1/\lambda\end{cases}\]
    We say it is {\em postcritically finite} if the forward orbit of the {\em critical point} $1-1/\lambda$ is finite.
\end{definition}

\begin{remark}
    A graph directed iterated function system with a single vertex and $n$ edges is an iterated function system with $n$ generators.
\end{remark}

    One can verify that the definition above is sufficient for running the argument by Solomyak and show that there are asymptotic self similarity around many points. Following the convention in \cite{lei1990similarity, solomyak2005mandelbrot}, we say a set $A$ in $\mathbb{C}$ is asymptotically self similar around at point $a$ if there is some $c\in\mathbb{C}$, $|c|>1$, such that the $\epsilon$-ball around $0$ of the sets $\{c^n(z-a): z\in A\}$ converges under Hausdorff topology. We say two sets $A$ and $B$ are asymptotically similar to one another around $a\in A$ and $b\in B$, if as $t\to\infty$, the Hausdorff distance between the $\epsilon$ ball around $0$ of the set $\{t(z-a): z\in A\}$ and the $\epsilon$ ball around $0$ of the set $\{t(z-b): z\in B\}$ converges to $0$. Given any infinite path $\gamma$ on $G$ going through consecutive edges $e_0, e_1, e_2, \dots$, denote 
    \[f_\gamma(z)=\lim_{n\to\infty}f_{e_0}(z, f_{e_1}(z, f_{e_2}(z,\dots f_{e_n}(z, 0)\cdot)))\]

    \begin{theorem}\label{jm}
        Suppose there is an infinite path $\gamma$ on $G$ of the form $\gamma_0\gamma_1^\infty$ but is not periodic, and $z\in\mathbb{C}$, $|c|<1$ is a simple root of the equation $f_\gamma(z)=C$, such that there are no other path $\gamma'$ on $G$ starting at $v_0$ satisfying $f_{\gamma'}(z)=C$. Then the Mandelbrot set is asymptotically self similar around $z$.
    \end{theorem}

    \begin{proof}
        Follow \cite{lei1990similarity}, what we need to check are the 4 assumptions of \cite[Proposition 4.1]{lei1990similarity}. Firstly we define the set 
        \[X=\{(z, x): |z|<1,  x=f_{\gamma'}(z)-f_{\gamma}(z), \text{ where }\gamma'\]
        \[\text{ is an infinite path on }G\text{ starting at }v_0\}\]
        Then Condition 1 is by construction, the sections in Condition 2 are 
        \[\{x=f_{\gamma'}(z)-f_{\gamma}(z): |z|<1\}\]
        Condition 4 follows from the fact that $z$ is a simple root. To show Condition 3, note that around $f_\gamma(w)$ there is a small neighborhood $U$ such that if $f_{\gamma'}(w)\in U$, then $\gamma'$ is of the form $\gamma_0\gamma_1\gamma''$. Now let $\Phi$ be the map sending this point to $f_{\gamma_0\gamma''}(w)$. Because $f_e$ are all holomorphic, and because $\gamma$ is the unique infinite path on $G$ such that $f_\gamma(z)=C$, when $w$ is sufficiently close to $z$, this map $\Phi$ is well defined and holomorphic. Let $c=\Phi'(f_\gamma(w))$.

        Now the theorem follows from the conclusion of \cite[Proposition 4.1]{lei1990similarity}.
    \end{proof}

    The rough idea is as follows: every point in the Mandelbrot set can be associated with one or more infinite paths on the directed graph that ends at the base vertex. If a point is associated with only a pair of paths of the form $\gamma^\infty \gamma'$, i.e. the concatenation of a periodic infinite path and a finite path, and there is a neighborhood in the Mandelbrot set where all points are associated with a path of the form $\gamma_1\gamma\gamma'$, then sending it to the point associated with $\gamma_1\gamma'$ is an asymptotic self similarity map.

Now our main result can be stated as follows:

\begin{theorem}\label{mainthm}
   Let $\lambda\in (\sqrt{2}, 2)$ be a number where the tent map $f_\lambda$ is postcritically finite, then the slice of $T$ at height $\lambda$ (by which we mean $\{z\in\mathbb{C}: (z, \lambda)\in T\}$) intersecting with the unit disc is the Mandelbrot set of a parametrized family of graph directed iterated function system, where the directed graph is the transition graph of a graph directed iterated function system, the directed graph $G$ is the transition graph of the Markov decomposition of $f_\lambda$ with all edges reversed, and the maps associated with the edges are either $(x, z)\mapsto zx+2-z$ or $(x, z)\mapsto z-zx$, and the constant $C=1$.
\end{theorem}

\begin{example}
    As an example, if $\lambda=\frac{\sqrt{5}+1}{2}$ is the golden ratio, then the graph $G$ has $2$ vertices $v_0$ and $v_1$ (with $v_0$ being the base vertex), and three edges $e_1$, $e_2$ and $e_3$, such that $e_1$ is from $v_0$ to $v_0$, associated with $(x, z)\mapsto z-zx$, $e_1$ is from $v_0$ to $v_1$, associated with $(x, z)\mapsto zx+2-z$, $e_2$ is from $v_1$ to $v_0$, associated with $(x, z)\mapsto z-zx$.
\end{example}

\begin{remark}
    The argument in \cite[Proposition 3.10]{bray2021shape} shows that $\lambda<\sqrt{2}$, then $(z, \lambda)\in T$ iff $(z^2, \lambda^2)\in T$. Hence, following Theorem \ref{jm}, there is asymptotic self similarities of slices at any height, not just those heights above $\sqrt{2}$.
\end{remark}

\section{The Criteria in \texorpdfstring{\cite{lindsey2023characterization}}{[3]}}

In this section we review the main results in \cite{lindsey2023characterization}.

\begin{definition}
The {\em twisted lexicographic order} on $\{0, 1\}^\mathbb{N}$ is defined as follows:
$a=(a_i)<b=(b_i)$ iff there is some $j$, such that $a_i=b_i$ for all $i<j$, and
\[(-1)^{\sum_{i<j}a_i}(a_j-b_j)<0\]
The twisted lexicographic order can also be defined for finite words of the same length analogously.
\end{definition}

\begin{definition}
Let $\sqrt{2}<\lambda<2$, let $f_\lambda$ be the tent map as in Definition \ref{tent}. Let $I_0=[0, 1-1/\lambda]$, $I_1=[1-1/\lambda, 1]$. For any $x\in [0, 1]$, an {\em itinerary} of $x$ under $f_\lambda$ is an infinite sequence $It_\lambda(x)=(a_0, a_1, a_2, \dots)$, such that $f^{\circ i}(x)\in I_{a_i}$. The {\em itinerary} of $f_\lambda$ is defined as
\[It_\lambda=\lim_{x\to 1^-, It_\lambda(x)\text{ is unique}}It_\lambda(x)\]
\end{definition}

\begin{remark}
    If the forward orbit of $x$ does not contain the critical point $1-1/\lambda$, $It_\lambda(x)$ is unique. 
\end{remark}

\begin{remark}\label{propadmissible}
A key observation of the Milnor-Thurston kneading theory \cite{milnor1988iterated} (see also \cite{bray2021shape}), is that $It_\lambda$ is {\em admissible}, i.e. it is no less that itself with any proper prefix removed under the twisted lexicographic order.
\end{remark}

\begin{definition}\cite[Definition 1.5]{lindsey2023characterization}\label{dfn:suitable}
Let $\sqrt{2}\leq \lambda<2$. A sequence $w\in \{0, 1\}^\mathbb{N}$ is called {\em $\lambda$-suitable}, if it satisfies the following three conditions:
\begin{enumerate}
    \item If $It_\lambda$ starts with $10^n1$, there can not be more than $n$ consecutive $1$s in $w$.
    \item For any $2>\lambda'>\lambda$, the reverse of any $n$ prefix of $w$ is no more than the n-prefix of $It_{\lambda'}$ under the twisted lexicographic order.
    \item When the reverse of the $n$ prefix of $w$ equals the $n$ prefix of $It_{\lambda'}$, the number of $1$s in this $n$-prefix is odd.
\end{enumerate}
Here the $n$-prefix of a sequence $w_0w_1w_2\dots$ is the finite word $w_0w_1\dots w_{n-1}$, the reverse of a finite word $w_0w_1\dots w_{n-1}$ is the word $w_{n-1}w_{n-2}\dots w_0$.
\end{definition}

The main theorem in \cite{lindsey2023characterization} is the following:

\begin{theorem}\cite[Theorem 1.7]{lindsey2023characterization}
   Let $\lambda\in [\sqrt{2}, 2)$, $|z|<1$. Then $(z, \lambda)\in T$ iff there is a sequence $w\in \{0, 1\}^\mathbb{N}$, such that
   \begin{enumerate}
       \item $w$ is $\lambda$-suitable.
       \item $\lim_{n\to\infty}f_{w_0, z}\circ f_{w_1, z}\circ \dots\circ f_{w_{n-1}, z}(z)=1$, where $f_{0, z}(x)=zx$, $f_{1, z}(x)=2-zx$.
   \end{enumerate}
\end{theorem}

\section{Proof of the Main Theorem}

\begin{proof}[Proof of Theorem \ref{mainthm}]
We first construct the parameterized family of graph directed iterated function systems. Consider the finite Markov decomposition of $f_\lambda$ formed by splitting interval $[0, 1]$ via the forward orbit of the critical point $1-1/\lambda$. Let $\Gamma$ be a directed graph, each vertex corresponding to a Markov block and there is an directed edge from  vertex $a$ to vertex $b$ iff the image of the interval corresponding to vertex $a$ under $f_\lambda$ contains the interval corresponding to vertex $b$. For every directed edge, if the starting vertex corresponds to a subinterval to the left of $1-1/\lambda$, associate it with the map $(x, z)\mapsto zx+2-z$; if the corresponding subinterval is to the right of $1-1/\lambda$, associate it with the map $(x, z)\mapsto z-zx$. We call the vertex corresponding to the subinterval containing critical value $1$ the base vertex $v_0$. To construct the family of graph directed iterated function systems, let $G$ be $\Gamma$ with all edges reversed, while keeping the base vertex and all the associated maps.

Now given any infinite path on $\Gamma$ ending at $v_0$ of the form
\[\dots\to v_2\to v_1\to v_0\]
Consider a 0-1 sequence $w=(w_0, w_1, \dots)$, such that $w_i=1$ if the subinterval corresponding to $v_i$ is to the right of $1-1/\lambda$, $w_i=0$ if it is to the left. Now the only thing remains to be shown is that such $w$ are exactly all the $\lambda$-suitable sequences in Definition \ref{dfn:suitable}.

\subsection{Show that $w$ are all $\lambda$-suitable}

From the definition of $\Gamma$, given any finite or infinite path on $\Gamma$, passing through vertices $a_0, a_1, \dots, a_n$, there is some $x$ in the interval such that $f_\lambda^j(x)$ is in the subinterval corresponding to vertex $a_j$. Hence from the construction of $w$, the reverse of any prefix of $w$ is a section of the itinerary of some point under $f_\lambda$. This implies Parts (1) and (2) of Definition \ref{dfn:suitable}. To show Part (3), suppose a finite word $\alpha$ which is the reverse of a prefix of $w$ is identical to a prefix of some $It_{\lambda'}$, then it is also identical to a prefix of $It_\lambda$. Now because $w$ comes from a path on $\Gamma$ ending at $v_0$, if we concatenate it with the path on $\Gamma$ where the vertices correspond to the subintervals containing $f_\lambda^j(1)$ (when there is ambiguity pick the subinterval corresponding to the definition $It_\lambda$, namely, the subinterval containing $f_\lambda^j(1-\epsilon)$ for some very small positive $\epsilon$), then we get an infinite path on $\Gamma$. Hence, we can find some $a\in [0, 1]$ such that $\alpha$ concatenated with $It_\lambda$ is an itinerary in $It_\lambda(a)$. 
If $a<1$, then $\alpha It_\lambda<It_\lambda$, which implies that the number of $1$s in $\alpha$ is odd. If $a=1$, let $\beta_\lambda$ be the shortest word such that $It_\lambda=\beta_\lambda^\infty$, $\beta'_\lambda$ be $\beta_\lambda$ with the last letter flipped from $0$ to $1$ and $1$ to $0$, then $\alpha It_\lambda$ is a concatenation of $\beta_\lambda$ or $\beta'_\lambda$. If $\alpha=\beta'_\lambda$ then the number of $1$s is odd. If $\alpha=\beta_\lambda$, suppose $j>0$ is the largest integer such that $It_\lambda'$ has a prefix of the form $\alpha^j$. Remove the common $\alpha^j$ prefix of $It_\lambda$ and $It_\lambda'$, we see that $It_\lambda'$ should have a prefix of the form $\alpha^j\gamma$ where $\gamma>\alpha$. Now compare $It_\lambda'$ and $It_\lambda'$ with the prefix $\alpha^j$ removed, we see that this contradicts with Remark \ref{propadmissible} hence is impossible.

\subsection{Show that \texorpdfstring{$\lambda$}{lambda}-suitable sequences all appear as some \texorpdfstring{$w$}{w}} We can recover the infinite path on $G$, where $v_i$ are the $i$-th vertex, from a $\lambda$-suitable sequence $w=(w_0, w_1, \dots)$ as follows: let $v_0$ be the vertex associated with the right most subinterval. For each $i$, if $w_i=0$, let $v_i$ be the subinterval to the left of the critical point whose image contains the subinterval associated with $v_{i-1}$, otherwise let $v_i$ be the subinterval to the right. $\lambda$-suitability now guaranteed that this procedure can continue indefinitely. If this is not the case and the procedure stops at $v_i$, then let $j$ be the smallest integer larger than $i$ such that $w_j=1$. The existence of such a $j$ is due to (1). Now take the $j+1$-th prefix of $w$ and reverse it, then by kneading theory, if both (2) and (3) are true, there must be some point on $[0, 1]$ with an itinerary where the $j+1$ prefix equals this finite $0-1$ word. This is a contradiction. \\ 

This finishes the proof of Theorem \ref{mainthm}.
\end{proof}

\bibliographystyle{amsplain}
\bibliography{refs}

\end{document}